\documentclass{amsart}

\usepackage{amssymb,
amsmath,
amsthm,
graphicx,
amscd,
multind,
eufrak
}
\theoremstyle{plain}
\newtheorem{thm}{Theorem}[section]
\newtheorem{lm}[thm]{Lemma}
\newtheorem{cor}[thm]{Corollary}
\newtheorem{prop}[thm]{Proposition}
\newtheorem{conj}[thm]{Conjecture}
\newtheorem{lem}[thm]{Lemma}
\theoremstyle{definition}
\newtheorem{de}[thm]{Definition}
\newtheorem{ex}[thm]{Example}
\newtheorem{re}[thm]{Remark}

\newcommand{\CC}{{\mathbb C}}
\newcommand{\RR}{{\mathbb R}}
\newcommand{\ZZ}{{\mathbb Z}}
\newcommand{\NN}{{\mathbb N}}

\newcommand{\im}{\operatorname{im}}
\newcommand{\supp}{\operatorname{supp}}
{\begin{figure} \begin{center}}%
{\end{center} \end{figure}}

\newcommand{\id}{\operatorname{id}}
\newcommand{\rk}{\operatorname{rk}}

\newcommand{\bu}{\mathbf{u}}
\newcommand{\la}{\langle}
\newcommand{\ra}{\rangle}
\newcommand{\Inc}{\operatorname{Inc}}
\newcommand{\Subs}{\operatorname{Subs}}
\newcommand{\GL}{\operatorname{GL}\nolimits}

\newcommand{\lieg}[1]{\mathrm{#1}}

\newcommand{\ot}{\leftarrow}

\newcommand{\Ainf}{A_\infty}
\newcommand{\Tinf}{T_\infty}
\newcommand{\Ginf}{G_\infty}
\newcommand{\Hinf}{H_\infty}
\newcommand{\Binf}{B_\infty}
\newcommand{\Dinf}{D_\infty}
\newcommand{\Sinf}{S_\infty}
\newcommand{\Xinf}[1][k]{X_\infty^{\leq #1}}
\newcommand{\Xp}[2][k]{X_{#2}^{\leq #1}}
\newcommand{\Yinf}[1][k]{Y_\infty^{\leq #1}}
\newcommand{\Yp}[2][k]{Y_{#2}^{\leq #1}}
\newcommand{\bw}{\mathbf{w}}
\newcommand{\bbb}{\mathfrak b}

\begin{document}

\begin{abstract}
Matrices of rank at most $k$ are defined by the vanishing of
polynomials of degree $k+1$ in their entries (namely, their $(k+1) \times
(k+1)$-subdeterminants), regardless of the size of the matrix. We prove a
qualitative analogue of this statement for tensors of arbitrary dimension,
where matrices correspond to two-dimensional tensors. More specifically,
we prove that for each $k$ there exists an upper bound $d=d(k)$ such
that tensors of {\em border rank} at most $k$ are defined by the vanishing
of polynomials of degree at most $d$, regardless of the dimension of
the tensor and regardless of its sizes in each dimension. Our proof
involves passing to an infinite-dimensional limit of tensor powers of
a vector space, whose elements we dub {\em
infinite-dimensional tensors},
and exploiting the symmetries of this limit in crucial way.
\end{abstract}

\title{Bounded-rank tensors are\\ defined in bounded degree}
\author[J.~Draisma]{Jan Draisma}
\address[Jan Draisma]{
Department of Mathematics and Computer Science\\
Technische Universiteit Eindhoven\\
P.O. Box 513, 5600 MB Eindhoven, The Netherlands\\
and Centrum voor Wiskunde en Informatica, Amsterdam,
The Netherlands}
\thanks{The first author is supported by a Vidi grant from
the Netherlands Organisation for Scientific Research (NWO)}
\email{j.draisma@tue.nl}

\author[J.~Kuttler]{Jochen Kuttler}
\address[Jochen Kuttler]{
Department of Mathematical and Statistical Sciences\\
632 Central Academic Building\\
University of Alberta\\
Edmonton, Alberta T6G 2G1,
CANADA}
\email{jochen.kuttler@ualberta.ca}

\maketitle

\tableofcontents

\section{Introduction}
Tensor rank is an important notion in such diverse
areas as multilinear algebra and algebraic geometry
\cite{Catalisano2002,Catalisano05b,Landsberg06}; algebraic statistics
\cite{Allman09,Garcia05}, in particular with applications to
phylogenetics \cite{Allman04,Casanellas05,Casanellas09,Sturmfels05b},
and complexity theory
\cite{Buergisser09,Buergisser97,Draisma09f,Landsberg06b,Landsberg08,Landsberg09,
Strassen69,Strassen83}. We will prove a fundamental, qualitative
property of tensor rank that we believe could have substantial impact
on theoretical questions in these areas. On the other hand, a future
computational counterpart to our theory may also lead to practical use
of the aforementioned property.

To set the stage, let $p,n_1,\ldots,n_p$ be natural numbers and let $K$
be an infinite field. An $n_1 \times \ldots \times n_p$-tensor $\omega$
(or $p$-tensor, for short) over $K$ is a $p$-dimensional array of scalars
in $K$. It is said to be {\em pure} if there exist vectors $x_1 \in
K^{n_1},\ldots,x_p \in K^{n_p}$ such that $\omega(i_1,\ldots,i_p)=x_1(i_1)
\cdots x_p(i_p)$ for all relevant indices. The {\em rank} of a general
tensor $\omega$ is the minimal number of terms in any expression of
$\omega$ as a sum of pure tensors. When $p$ equals $2$, this notion
of rank coincides with that of matrix rank. As a consequence, linear
algebra provides an alternative characterisation of rank for $2$-tensors:
the $2$-tensors of rank at most $k$ are precisely those for which all
$(k+1)\times(k+1)$-subdeterminants vanish. When $p$ is larger than $2$,
no such clean alternative characterisation seems to exist. One immediate
reason for this is that having rank at most $k$ is typically not a
Zariski-closed condition, that is, not equivalent to the vanishing of
some polynomials; the following classical example illustrates this.

\begin{ex}
Take $K=\CC$, $p=3$, and $n_1=n_2=n_3=2$. A tensor
$\omega$ defines a linear map $\phi$ from $\CC^2$ to the space $M_2(\CC)$
of $2 \times 2$-matrices over $\CC$ by
\[
\phi:(x_1,x_2) \mapsto
\begin{bmatrix}
	x_1\omega(1,1,1)+x_2\omega(1,1,2) &
	x_1\omega(1,2,1)+x_2\omega(1,2,2) \\
	x_1\omega(2,1,1)+x_2\omega(2,1,2) &
	x_1\omega(2,2,1)+x_2\omega(2,2,2)
\end{bmatrix}. \]
If $\im \phi$ is zero, then $\omega$ has rank $0$. If $\im \phi$
is one-dimensional, then the rank of $\omega$ equals the rank of any
non-zero matrix in $\im \phi$. If $\im \phi$ is spanned by two linearly
independent matrices of rank $1$, then the rank of $\omega$ is two. This
is the ``generic'' situation. In the remaining case, the rank is $3$. The
$3$-tensors of the latter type are contained and dense in the zero
set of Cayley's {\em hyperdeterminant}. This quartic polynomial
in the entries of $\omega$ equals the discriminant of determinant of
$\phi(x)$, seen as a homogeneous quadratic polynomial in $x_1,x_2$ with
coefficients that are homogeneous quadratic polynomials in the entries of
$\omega(i,j,k)$. Hence the set of tensors of rank at most $2$ is dense
in all $2 \times 2 \times 2$-tensors, and its complement is non-empty
and locally closed. Moreover, working over $\RR$ instead of $\CC$, the
tensors whose hyperdeterminant is negative also have rank $3$.
This serves to show that tensor rank is a delicate
notion. 
\end{ex}

This situation leads to the notion of {\em border rank} of a tensor
$\omega$, which is the minimal $k$ such that all polynomials that vanish
identically on the set of tensors of rank at most $k$ also vanish on
$\omega$. In other words, $\omega$ has border rank at most $k$ if it
lies in the Zariski-closure of the set of tensors of rank at most $k$.
Thus all $2 \times 2 \times 2$-tensors over $\CC$ have border rank
at most $2$, and the same is true over $\RR$. Furthermore, over $\CC$
the tensors of border rank at most $k$ is also the ordinary (Euclidean)
closure of the set of tensors of rank at most $k$. This explains the
term border rank: tensors of border rank at most $k$ are those that can
be approximated arbitrarily well by tensors of rank at most $k$.

There is an important situation where rank and border rank do agree, and
that is the case of rank at most one, i.e., pure tensors. These are always
defined by the vanishing of all $2 \times 2$-subdeterminants of {\em
flattenings} of the $p$-tensor $\omega$ into matrices.  Such a flattening
is obtained by partitioning the $p$ dimensions $n_1,\ldots,n_p$ into two
sets, for example $n_1,\ldots,n_q$ and $n_{q+1},\ldots,n_p$, and viewing
$\omega$ as a $(n_1 \cdots n_q) \times (n_{q+1}\cdots n_p)$-matrix. In
fact, it is well known that the $2 \times 2$-subdeterminants of all
flattenings even generate the ideal of all polynomials vanishing on
pure tensors.

Among the most advanced results on tensor rank is the fact that a
$p$-tensor has border rank at most $2$ if and only if all $3 \times
3$-subdeterminants of all its flattenings have rank at most $2$
\cite{Landsberg04}; the {\em GSS-conjecture} in \cite{Garcia05} that these
subdeterminants generate the actual ideal has recently been proved
by Raicu \cite{Raicu10}. In any case, having border rank at most $2$
is a property defined by polynomials of degree $3$, regardless of the
dimension of the tensor. Replacing $2$ by $k$ and $3$ by ``some degree
bound'', we arrive the statement of our first main theorem, to which
the title of this paper refers.

\begin{thm}[Main Theorem I]
For fixed $k \in \NN$, there exists a $d \in \NN$ such that for all $p \in
\NN$ and $n_1,\ldots,n_p \in \NN$ the set of $n_1 \times \ldots \times
n_p$ tensors over $K$ of border rank at most $k$ is defined by the vanishing
of a number of polynomials of degree at most $d$.
\end{thm}

Table \ref{tab:dk} shows what is known about the smallest $d(k)$
satisfying the conclusion of the theorem for $K$ of characteristic
$0$.

\begin{table}
\begin{tabular}{l|l|l|l|l|l|l|l}
$k$ & $0$ & $1$ & $2$ & $3$ & $4$ & $(3l-1)/2$ ($l \geq 3$ odd) & $k$ \\
\hline
$d(k)$ & $1$ & $2$ & $3$ & $\geq 4$ & $\geq 9$ & $\geq 3l$ & $\geq k+1$ \\
\end{tabular}\\
\ \\
\ \\
\caption{Degree bounds: $d(0)=1$ is trivial; $d(1)=2$ is classical---$2
\times 2$-subdeterminants of flattenings generate the ideal of pure
tensors; $d(2)=3$ is due to \cite{Landsberg04} and scheme-theoretically
to \cite{Raicu10}; for $k=3$ there are already quartic polynomials in
the ideal for $3 \times 3 \times 3$-tensors (while there are no $4 \times
4$-subdeterminants of flattenings) \cite[Proposition 6.1]{Landsberg04};
these are due to Strassen \cite[Section 4]{Strassen83} and together
they generate the ideal if $p=3$ \cite[Theorem 1.3]{Landsberg06} but it
is unknown whether they define border-rank $\leq 3$ tensors for larger
$p$; $d((3l-1)/2) \geq 3l$ and its special case $d(4) \geq 9$ follows
from Strassen's degree-$3l$ equation for the hypersurface of $l \times
l \times 3$-tensors of border rank at most $(3l-1)/2$ \cite[Theorem
4.6]{Strassen83} (see also \cite{Landsberg08}); and $d(k) \geq
k+1$ follows from a general fact about ideals of secant varieties
\cite[Corollary 3.2]{Landsberg04} applied to the variety of pure tensors.}
\label{tab:dk}
\end{table}

\begin{re} 
{\em A priori}, the bound $d$ in Main Theorem I depends on the infinite
field $K$, as well as on $k$.  But take $L$ to be any algebraically
closed field of the same characteristic as $K$. Then, by Hilbert's Basis
theorem, $d$ from Main Theorem I applied to $L$ has the property that for
all $p,n_1,\ldots,n_p$ the ideal of all polynomials vanishing on tensors
over $L$ of rank at most $k$ is the radical of its sub-ideal generated
by elements of degree at most $d$. This latter statement (certain ideals
given as kernels of $L$-algebra homomorphisms defined over $\ZZ$ are the
radicals of their degree-at-most-$d$ parts) can be phrased as a linear
algebra statement over the common prime field of $L$ and $K$. Hence any
$d$ that works for $L$ also works for $K$. This shows the existence of a
degree bound that depends only on $k$ and on the {\em characteristic} of
$K$. It seems likely that a uniform bound exists that does not depend on
the characteristic either, but that would require more subtle arguments.
\end{re}

An interesting consequence of Main Theorem I is the following, for
whose terminology we refer to \cite{Allman04}.

\begin{cor} \label{cor:GMM}
For every fixed $k$, there exists a uniform degree bound on equations
needed to cut out the $k$-state general Markov model on any phylogenetic
tree (in which any number of taxa and of internal vertices
of any degree are allowed).
\end{cor}

Indeed, in \cite{Allman04} it is proved that the $k$-state general Markov
model on any tree is cut out by certain $(k+1) \times (k+1)$-determinants
together with the equations coming from star models (and this also holds
scheme-theoretically \cite{Draisma07b}). These star models, in turn, are
varieties of border-rank-$k$ tensors. Note that set-theoretic equations
for the $4$-state model were recently found in \cite{Friedland10} and,
using numerical algebraic geometry, in \cite{Bates10}.

This paper is organised as follows. Section~\ref{sec:Tensors}
recalls a coordinate-free version of tensors, together with the
fundamental operations of contraction and flattening, which do
not increase (border) rank. That section also recalls a known
reduction to the case where all sizes $n_i$ are equal to any fixed
value $n$ greater than $k$ (Lemma~\ref{lm:AllSame}), as well as
an essentially equivalent reformulation of Main Theorem I (Main
Theorem II). In Section~\ref{sec:Infinite} we introduce a natural
projective limit $\Ainf$ of the spaces of $n \times n \times \ldots
\times n$-tensors (with the number of factors tending to infinity),
which contains as an (infinite-dimensional) subvariety the set
$\Xinf$ of {\em $\infty$-dimensional tensors of border rank at most
$k$}. This variety is contained in the {\em flattening
variety} $\Yinf$ defined by all
$(k+1)\times(k+1)$-subdeterminants of flattenings. In Section~\ref{sec:Yk}
we prove that $\Yinf$ is defined by finitely many orbits of equations
under a group $\Ginf$ of natural symmetries of $\Ainf,\Xinf,\Yinf$
(Proposition~\ref{prop:YFinite}). In Section~\ref{sec:EqNoeth} we
recall some basic properties of equivariantly Noetherian topological
spaces, and prove that $\Yinf$ is a $\Ginf$-Noetherian topological
space with the Zariski topology (Theorem~\ref{thm:YEqNoether}). This
means that any $\Ginf$-stable closed subset is defined by finitely many
$\Ginf$-orbits, and hence in particular this holds for $\Xinf$ (Main
Theorem III below). In Section~\ref{sec:Proofs} we show how this implies
Main Theorems I and II, and in Section~\ref{sec:Ideal} we speculate on
finiteness results for the entire ideal of $\Xinf$, rather than just
set-theoretic finiteness.

\section*{Acknowledgments}
We thank Sonja Petrovi\'c for pointing out that the formulation of
Corollary~\ref{cor:GMM} in an earlier version of this paper was incorrect.

\section{Tensors, flattening, and contraction} \label{sec:Tensors}

For most of the arguments in this paper---with the notable exception of
the proof of Lemma~\ref{lm:Zprime}---a coordinate-free notion of tensors
will be more convenient than the array-of-numbers interpretation from the
Introduction. Hence let $V_1,\ldots,V_p$ be finite-dimensional vector
spaces over a fixed infinite field $K$.  Setting $[p]:=\{1,\ldots,p\}$,
we write $V_I:=\bigotimes_{i \in I} V_i$ for the tensor product of the
$V_i$ with $i \in I \subseteq [p]$. The rank of a tensor $\omega$ in
$V_{[p]}$ is the minimal number of terms in any expression of $\omega$
as a sum of pure tensors $\bigotimes_{i \in [p]} v_i$ with $v_i \in V_i$.

Given a $p$-tuple of linear maps $\phi_i: V_i \to U_i$,
where $U_1,\ldots,U_p$ are also vector spaces over $K$, we write
$\phi_{[p]}:=\bigotimes_{i \in [p]}\phi_i$ for the linear map $V_{[p]} \to
U_{[p]}$ determined by $\bigotimes_{i \in [p]} v_i \mapsto \bigotimes_{i
\in [p]} \phi_i(v_i)$. Clearly $\rk \pi_{[p]} \omega \leq \rk \omega$ for
any $\omega \in V_{[p]}$, and this inequality carries over to the border
rank. In the particular case where a single $U_j$ equals $K$, $U_i$ equals
$V_i$ for all $i \neq j$, and $\phi_i=\id_{V_i}$ for all $i \neq j$, we
identify $U_{[p]}$ with $V_{[p] - \{j\}}$ in the natural way, and we
call the map $\phi_{[p]}:V_{[p]} \to V_{[p]-\{j\}}$ the {\em contraction}
along the linear function $\phi_j$. The composition of contractions
along linear functions $\phi_j$ on $V_{j}$ for all $j$ in some subset $J$
of $[p]$ is a linear map $V_{[p]} \to V_{[p]-J}$, called a contraction
along the pure $|J|$-tensor $\bigotimes_{j \in J} \phi_j$, that does not
increase (border) rank. Now we can phrase a variant of our Main Theorem I.

\begin{thm}[Main Theorem II.]
For all $k \in \NN$ there exists a $p_0$ such that for all $p \geq p_0$
and for all vector spaces $V_1,\ldots,V_p$ a tensor $\omega \in V_1
\otimes \cdots \otimes V_p$ has border rank at most $k$ if and only if
all its contractions along pure $(p-p_0)$-tensors have border rank at
most $k$.
\end{thm}

\begin{re}
If $p_0$ is known explicitly, this theorem gives rise to the following
theoretical randomised algorithm for testing whether a given $p$-tensor
has border rank at most $k$: for each $(p-p_0)$-subset $J$ of $[p]$
contract along a random pure tensor in the dual of $V_J$, and on the
resulting $p_0$-tensor evaluate the polynomials cutting out tensors of
border rank at most $k$. If each of these tests is positive, then one
can be fairly certain that the given tensor is, indeed, of border rank
at most $k$. While this might not be a practical algorithm, its number
of arithmetic operations is polynomial in the number of entries in the
tensor, as long as we take $k$ constant.
\end{re}

In our proofs of Main Theorems I and II we will use a second operation
on tensors, namely, {\em flattening}. Suppose that $I,J$ form a
partition of $[p]$ into two parts. Then there is a natural isomorphism
$\flat=\flat_{I,J}: V_{[p]} \to V_I \otimes V_J$. The image $\flat
\omega$ is a $2$-tensor called a flattening of $\omega$. Its rank (as
a $2$-tensor) is a lower bound on the border rank of $\omega$.
The first step in our proof below is a reduction to the case where all
$V_i$ have the same dimension.

\begin{lm} \label{lm:AllSame}
Let $p,k,n$ be natural numbers with $n \geq k+1$, and let $V_1,\ldots,V_p$
be vector spaces over $K$. Then a tensor $\omega \in \bigotimes_{i \in
[p]} V_i$ has rank (respectively, border rank) at most $k$ if and only
if for all $p$-tuples of linear maps $\phi_i:V_i \to K^n$ the tensor
$(\bigotimes_{i \in [p]}\phi_i)\omega$ has rank (respectively, border
rank) at most $k$.
\end{lm}

This lemma is well known, and even holds at the scheme-theoretic level
\cite[Theorem 11]{Allman04}. We include a proof for completeness.

\begin{proof}
The ``only if'' part follows from the fact that $\phi_{[p]}$ does
not increase rank or border rank. For the ``if'' part assume that
$\omega$ has rank strictly larger than $k$, and we argue that there
exist $\phi_1,\ldots,\phi_p$ such that $\phi_{[p]} \omega$ still has
rank larger than $k$. It suffices to show how to find $\phi_1$; the
remaining $\phi_i$ are found in the same manner. Let $U_1$ be the image
of $\omega$ regarded as a linear map from the dual space $V_{[p]-\{1\}}^*$
to $V_1$. If $U_1$ has dimension at most $n$, then by elementary linear
algebra there exist linear maps $\phi_1:V_1 \to K^n$ and $\psi_1:K^n \to
V_1$ such that $\psi_1 \circ \phi_1$ is the identity map on $U_1$. Set
$\omega':=\phi_1 \otimes (\bigotimes_{i>1} \id_{V_i})\omega$, so that by
construction $\omega$ itself equals $\psi_1 \otimes (\bigotimes_{i>1}
\id_{V_i}) \omega'$. By the discussion above, we have the inequalities
$\rk \omega \geq \rk \omega' \geq \rk \omega$, so that both ranks are
equal and larger than $k$, and we are done. If, on the other hand, $U_1$
has dimension larger than $n$, then let $\phi_1:V_1 \to K^n$ be any
linear map that maps $U_1$ surjectively onto $K^n$. Defining $\omega'$
as before, we find that the image of $\omega'$ regarded as a linear
map $V_{[p]-\{1\}}^* \to K^n$ is all of $K^n$, so that the flattening
$\flat_{1,[p]-\{1\}} \omega'$ already has rank $n>k$. This implies that
$\omega'$ itself has rank larger than $k$.
\end{proof}

\section{Infinite-dimensional tensors and their symmetries}
\label{sec:Infinite}

The space of $\NN \times \NN$-matrices can be realised as the projective
limit of the spaces of $p \times p$-matrices for $p \to \infty$, where
the projection from $(p+1) \times (p+1)$-matrices to $p \times p$-matrices
sends a matrix to the $p \times p$-matrix in its upper left corner. In
much the same way, we will construct a space of {\em infinite-dimensional
tensors} as a projective limit of all tensor powers $V^{\otimes p},\
p \in \NN$ of the vector space $V=K^n$ for some fixed natural number
$n$. To this end, fix a linear function $x_0 \in V^*$, and denote by $\pi_p$
the contraction $V^{\otimes p} \to V^{\otimes p}$ determined by
\[ \pi_p: v_1 \otimes \cdots \otimes v_p \otimes v_{p+1}
\mapsto
x_0(v_{p+1}) \cdot v_1 \otimes \cdots \otimes v_p. \]
Dually, this surjective map gives rise to the injective linear map
\[ (V^*)^{\otimes p} \to (V^*)^{\otimes p+1},\
\xi \mapsto \xi \otimes x_0. \]
Let $T_p$ be the coordinate ring of $V^{\otimes p}$. We identify $T_p$
with the symmetric algebra $S((V^*)^{\otimes p})$ generated by the
space $(V^*)^{\otimes p}$, and embed $T_p$ into $T_{p+1}$ by means of
the linear inclusion $(V^*)^{\otimes p} \to (V^*)^{\otimes p+1}$ above.
Now the projective limit $\Ainf:=\lim_{\ot p} V^{\otimes p}$ along the
projections $\pi_p$ is a vector space whose coordinate ring $\Tinf$
is the union $\bigcup_{p \in \NN} T_p$. Here by coordinate ring we mean that $\Ainf$ corresponds naturally (via evaluation) to the set of $K$-valued points of $\Tinf$ (ie.\ $K$-algebra homomorphisms $T \to K$). Not every vector space has a coordinate ring in that sense, as it is equivalent to being the dual space of a vector space: $\Ainf$ is naturally the dual space of the space ${\Tinf}^1$ of homogeneous polynomials of degree $1$ in $\Tinf$. For example, as ${\Tinf}^1$ has a countably infinite basis, it is not a dual space and has itself no coordinate ring in that sense.

At a crucial step in our arguments we will use the following more concrete
description of $\Tinf$. Extend $x_0$ to a basis $x_0,x_1,\ldots,x_{n-1}$
of $V^*$, so that $(V^*)^{\otimes p}$ has a basis in bijection with
words $w=(i_1,\ldots,i_p)$ over the alphabet $\{0,\ldots,n-1\}$, namely,
$x_w:= x_{i_1} \otimes \cdots \otimes x_{i_p}$.  The algebra $T_p$ is
the polynomial algebra in the variables $x_w$ with $w$ running over all
words of length $p$. In $\Tinf$, the coordinate $x_w$ is identified with
the variable $x_{w'}$ where $w'$ is obtained from $w$ by appending a $0$
at the end. We conclude that $\Tinf$ is a polynomial ring in countably
many variables that are in bijective correspondence with infinite words
$(i_1,i_2,\ldots)$ in which all but finitely many characters are $0$. The
finite set of positions where such a word is non-zero is called the {\em
support} of the word.

Now for symmetry: the symmetric group $S_p$ acts on $V^{\otimes p}$
by permuting the tensor factors:
\[ \pi (v_1 \otimes \cdots \otimes v_p) =
v_{\pi^{-1}(1)} \otimes \cdots \otimes v_{\pi^{-1}(p)}. \]
This leads to the contragredient action of $S_p$ on the dual space
$(V^*)^{\otimes p}$ by
\[ \pi(x_1 \otimes \cdots \otimes x_p) =
x_{\pi^{-1}(1)} \otimes \cdots \otimes x_{\pi^{-1}(p)}, \]
which extends to an action of $S_p$ on all of $T_p$ by means of algebra
automorphisms.  In terms of words, $\pi x_w$ equals the coordinate
$x_{w'}$ where the $q$-th character of $w'$ is the $\pi^{-1}(q)$-th
character of $w$. Let $\Sinf$ denote the union $\bigcup_{p \in \NN}
S_p$, where $S_p$ is embedded in $S_{p+1}$ as the subgroup fixing $p+1$.
Hence $\Sinf$ is the group of all bijections $\pi: \NN \to \NN$ whose
set of fixed points has a finite complement. This group acts on $\Ainf$
and on $\Tinf$ by passing to the limit.

The action of $\Sinf$ on $\Tinf$ has the following fundamental property:
for each $f \in \Tinf$ there exists a $p \in \NN$ such that whenever
$\pi,\sigma \in \Sinf$ agree on the initial segment $[p]$ we have
$\pi f=\sigma f$. Indeed, we may take $p$ equal to the maximum of the
union of the supports of words $w$ for which $x_w$ appears in $f$. In
this situation, there is a natural left action of the {\em increasing
monoid} $\Inc(\NN)=\{\pi: \NN \to \NN \mid \pi(1)<\pi(2)<\ldots\}$ by
means of injective algebra endomorphisms on $\Tinf$; see \cite[Section
5]{Hillar09}. The action is defined as follows: for $f \in \Tinf$,
let $p$ be as above. Then to define $\pi f$ for $\pi \in \Inc(\NN)$
take any $\sigma \in \Sinf$ that agrees with $\pi$ on the interval $[p]$
(such a $\sigma$ exists) and set $\pi f:=\sigma f$. In terms of words:
$\pi x_w$ equals the coordinate $x_{w'}$, where the $q$-th character of
$w'$ is the $\pi^{-1}(q)$-th character of $w$ if $q \in \im \pi$, and $0$
otherwise. By construction, the $\Inc(\NN)$-orbit of any $f \in \Tinf$
is contained in the $\Sinf$-orbit of $f$. Note that the left action of
$\Inc(\NN)$ on $\Tinf$ gives rise to a {\em right} action of $\Inc(\NN)$
by means of surjective linear maps $\Ainf \to \Ainf$. Since we cannot
take inverses, there is no corresponding left action on $\Ainf$, while
of course the left action of $\Sinf$ can be made into a right action by
taking inverses.  A crucial argument in Section~\ref{sec:EqNoeth} uses
a map that is not equivariant with respect to $\Sinf$ but is equivariant
relative to $\Inc(\NN)$.

Apart from $S_p$, the group $\lieg{GL}(V)^p$ acts linearly on $V^{\otimes p}$ by
\[ (g_1,\ldots,g_p)(v_1 \otimes \cdots v_p)=
(g_1 v_1 \otimes \cdots \otimes g_p v_p), \]
and this action gives a right action on $(V^*)^{\otimes p}$ by
\[ (z_1 \otimes \cdots \otimes z_p)(g_1,\ldots,g_p)
((z_1 \circ g_1) \otimes \cdots \otimes (z_p \circ g_p)). \]
Regarding $\lieg{GL}(V)^p$ as the subgroup of $\lieg{GL}(V)^{p+1}$
consisting of all tuples $(g_1,\ldots,g_p,1)$, we obtain a (left) action
of the union of all $\lieg{GL}(V)^p$ on $\Ainf$, as well as an action
on $\Tinf$ by means of algebra automorphisms.

The group generated by $S_p$ and $\lieg{GL}(V)^p$ in their representations
on $V^{\otimes p}$ is the semidirect (wreath) product $S_p \ltimes
\lieg{GL}(V)^p$, which we denote by $G_p$. The embeddings $S_p \to
S_{p+1}$ and $\lieg{GL}(V)^p \to \lieg{GL}(V)^{p+1}$ together determine
an embedding $G_p \to G_{p+1}$, and the union $\Ginf$ of all $G_p,\ p \in
\NN$ has an action on $\Ainf$ and on $\Tinf$ by means of automorphisms.

Now we come to tensors of bounded border rank. Write $\Xp{p}$ for the
subvariety of $V^{\otimes p}$ consisting of tensors of border rank at
most $k$. Since $\pi_p$ is a contraction along a pure tensor, it maps
$\Xp{p+1}$ into $\Xp{p}$; let $\Xinf \subseteq \Ainf$ be the projective
limit of the $\Xp{p}$. The elements of $G_p$ stabilise $\Xp{p}$, and
hence the elements of $\Ginf$ stabilise $\Xinf$.  We can now state our
third Main Theorem, which will imply Main Theorems I and II.

\begin{thm}[Main Theorem III]
For any fixed natural number $k$, the set $\Xinf$ is the common zero
set of finitely many $\Ginf$-orbits of polynomials in $\Tinf$.
\end{thm}

We emphasise here that Main Theorem III only concerns the variety $\Xinf$
as a set of $K$-valued points. We do not claim that the {\em ideal}
of $\Xinf$ is generated by finitely many $\Ginf$-orbits of polynomials;
see Section~\ref{sec:Ideal}. 

\section{The flattening variety} \label{sec:Yk}
%NOTE: Subspace variety corresponds to partitioning in singleton +
%rest.

In this section we discuss the {\em flattening variety} $\Yinf \subseteq \Ainf$,
which contains $\Xinf$ and is defined by explicit equations. These
equations are determinants obtained as follows. Given any partition of
$[p]$ into $I,J$ we have the flattening $V^{\otimes p} \to V^{\otimes I}
\otimes V^{\otimes J}$. Composing this flattening with a $(k+1) \times
(k+1)$-subdeterminant of the resulting two-tensor gives a degree-$(k+1)$
polynomial in $T_p$. The linear span of all these equations for all
possible partitions $I,J$ is a $G_p$-submodule $U_p$ of $T_p$, which by
$\pi_p^*$ is embedded into $U_{p+1}$. Let $\Yp{p}$
 denote the subvariety of $V^{\otimes p}$ defined by
the $U_p$. This is a $G_p$-stable variety, and $\pi_p$ maps $\Yp{p+1}$
into $\Yp{p}$. The projective limit $\Yinf \subseteq \Ainf$ is the zero
set of the space $U:=\bigcup_p U_p$. The variety $\Yinf$ consists of all
infinite-dimensional tensors all of whose flattenings to $2$-tensors have
rank at most $k$. In particular, since flattening does not increase rank,
$\Yinf$ contains $\Xinf$.

For later use we describe these determinants in more concrete terms in
the coordinates $x_w$, where $w$ runs through the infinite words over
$\{0,1,\ldots,n-1\}$ of finite support. Let $\bw=(w_1,\ldots,w_l)$ be an
$l$-tuple of pairwise distinct infinite words with finite support. Let
$\bw':=(w'_1,\ldots,w'_m)$ be another such tuple, and assume that the
support of each $w_i$ is disjoint from that of each $w'_j$. Then we
let $x[\bw;\bw']$ be the $l \times m$ matrix with $(i,j)$-entry equal
to $x_{w_i+w'_j}$. The space $U$ is spanned by the determinants of all
matrices $x[\bw;\bw']$ with $l=m=k+1$.

\begin{ex} \label{ex:Det}
Take $w_1=1,w_2=2,w'_1=01,w'_2=02$, where we follow the convention
that the (infinitely many) trailing zeroes of words may be left out,
and where we have left out brackets and separating commas for brevity.
Then
\[ x[\bw;\bw']=
\begin{bmatrix}
x_{11} & x_{12}\\
x_{21} & x_{22}\\
\end{bmatrix}. \]
\end{ex}

Apart from the inclusion $\Xinf \subseteq \Yinf$, we will need the
following crucial property of $\Yinf$, whose proof will take up the
remainder of this section.

\begin{prop} \label{prop:YFinite}
For every fixed natural number $k$ the flattening variety $\Yinf$ is
the common zero set of finitely many $\Ginf$-orbits of $(k+1) \times
(k+1)$-determinants $\det x[\bw;\bw']$ with $\bw,\bw'$ as above.
\end{prop}

The proposition is an immediate consequence of the following related result:
\begin{lem}\label{lem:ContractionPP_0}
Let $k \geq 0$ be fixed.  There is an integer $p_0$ such that whenever
$p > p_0$ and whenever $v\in  V^{\otimes p}$  is an element such that
for all pure contractions $\phi \colon V^{\otimes p} \to V^{\otimes
[p]-I}\cong V^{\otimes p_0}$ with $\vert I \vert = p-p_0$, $\phi(v)
\in \Yp{p_0}$, then $v\in \Yp{p}$.
\end{lem}

\begin{proof}[Proof of Proposition~\ref{prop:YFinite}]
Let $p_0$ be the integer whose existence is asserted by Lemma~\ref{lem:ContractionPP_0}.
Let $f_1,f_2,\dots,f_N \in T_{p_0}$ be the finitely many $(k+1) \times (k+1)$-determinants that generate the ideal of $\Yp{p_0}$. Of course, in the inclusion $T_{p_0} \subset \Tinf$, each $f_i$ is one of the $\det x[\bw;\bw']$ for $\bw,\bw'$
each lists of $k+1$ words supported in $[p_0]$.

We will now show that $v \in \Ainf$ is an element of $\Yinf$ if and only if
$f_i(gv) = 0$ for all $i$ and all $g \in \Ginf$.  Note that $f_i(gv)$ is equal
to $f_i((gv)_{p_0})$ where $(gv)_{p_0}$ is the image of $v$ in $V^{\otimes
p_0}$ under the canonical projection $\Ainf \to V^{\otimes p_0}$. Now if $v \in
\Yinf$, then obviously so is $gv$ for each $g \in \Ginf$, and hence
$(gv)_{p_0}$ is an element of $\Yp{p_0}$.
This shows the only if part.

For the converse, suppose that $f_i(gv) = 0$ for all $i$ and all $g
\in \Ginf$. We need to show that $v \in \Yinf$. Equivalently, we need
to show that for all $p$, the image $v_p \in V^{\otimes p}$ of $v$ lies
in $\Yp{p}$.  Recall that $f_i \in T_{p_0}$ is identified in $T_p$ with
$f_i$ precomposed with the pure contraction $\phi$ of the last $p-p_0$
factors along $x_0^{\otimes(p-p_0)}$.  Now if $\phi'\colon V^{\otimes p}
\to V^{\otimes p_0}$ is any other (nonzero) pure contraction, then there
is $g \in G_p$ such that $\phi'(v) = \phi(gv)$ for all $v \in V^{\otimes
p}$; indeed, the symmetric group $S_p$ can be used to ensure that the
same factors are being contracted, and $\GL(V)^p$ can be used to ensure
that the contraction takes place along the same pure tensor.  If $g \in
G_p$, then $(gv)_p = gv_p$, and it follows that our assumption implies
that $f_i(\phi'(v_p)) = 0$ for all $i$ and all pure contractions $\phi'
\colon V^{\otimes p} \to V^{\otimes p_0}$. By the lemma this means
that $v_p \in \Yp{p}$, and we are done.
\end{proof}

The proof of the lemma itself---while not difficult---requires some
algebraic geometry.  As it will turn out, the crucial point is the
following observation.

\begin{lem}\label{lem:SubspaceContraction}
Let $k \geq 0$. There is an integer $p_1$ such that whenever $p > p_1$
and $W \subseteq V^{\otimes p}$ is a subspace the following holds:
if for all $i$ and all pure contractions $\phi \colon V^{\otimes p}
\to V^{\otimes [p]-\{i\}}$, $\dim \phi(W) \leq k$, then $\dim W \leq k$.
\end{lem}

\begin{proof}
We will show that $p_1 = (k+1)\lfloor \log_2(k+1) \rfloor$ works.
Let $G(d,V^{\otimes p})$ denote the Grassmannian of $d$-planes in
$V^{\otimes p}$, which is a projective algebraic variety over $K$. Set
\[
Z(d,k) := \{ W \in G(d,V^{\otimes p}) \mid \dim \phi(W) \leq k \text{
for all contractions } \phi \},
\]
a closed subvariety of $G(d,V^{\otimes p})$. 
The assertion of the lemma is equivalent to the statement that $Z(d,k)
= \emptyset$ if $d > k$ and $p>p_1$. To prove this, we proceed by
induction on $d$. If $W \in Z(d,k)$, then all hyperplanes of $W$ are
elements of $Z(d-1,k)$, hence if the latter variety is empty, then
so is $Z(d,k)$. Thus, in the following we may assume that $d = k + 1$
and force a contradiction.

So suppose that $Z(k+1,k)$ is nonempty. We will use that it is
stable under $\GL(V)^p \subseteq G_p$. Denote by $B$ the group
of upper triangular matrices in $\GL(V)$ with respect to the basis
$e_0,\ldots,e_{n-1}$ of $V$ dual to $x_0,\ldots,x_{n-1}$.  Then $B^p$
is a connected solvable algebraic group and by Borel's Fixed Point
Theorem (\cite{Borel91}, Theorem~15.2), $B^p$ must have a fixed point $W$
on the projective algebraic variety $Z(d,k)$. This means that $W$ is a
$B^p$-stable subspace of $V^{\otimes p}$.

Let $D \subseteq \GL(V)$ denote the subset of diagonal matrices. Then
$D \subseteq B$ and hence $W$ is also $D^p$-stable. Any $D^p$-stable
subspace is spanned by common eigenvectors for $D^p$ (any algebraic
representation of $D^p$ is diagonalisable). Now $v \in V^{\otimes p}$ is
a $D^p$-eigenvector if and only if $v = e_{i_1} \otimes e_{i_2} \otimes
\dots \otimes e_{i_p}$ (up to a nonzero scalar) for some $i_1,\ldots,i_p$
with $0\leq i_k \leq n-1$. If such a $v$ is in $W$, then also the linear
span of $B^pv$ is. We claim that
\[ \la B^p v\ra = E_{i_1} \otimes E_{i_2} \otimes \dots \otimes E_{i_p}
= : E \]
where $E_i = \la e_0,e_1,\dots,e_i\ra \subseteq V$ is the $B$-stable
subspace generated by $e_i$.  Notice that $E$ is clearly $B^p$-stable. The
only non-trivial assertion is therefore that $\la B^p v\ra$ contains $E$.
This is seen most easily by considering the action of the Lie algebra
$\bbb^p$ of $B^p$ where $\bbb$ is the Lie algebra of $B$, ie.\ the
space of all upper triangular matrices. Since $B^p$ maps $W$ to itself,
so does $\bbb^p$.  An element $X = (X_1,X_2,\dots,X_p)$ of $\bbb^p$
acts on $V^{\otimes p}$ by
\[ X(v_1\otimes v_2 \otimes \dots \otimes v_p) = \sum_{i=1}^p v_1\otimes v_2 \otimes \dots \otimes v_{i-1} \otimes X_i(v_i) \otimes v_{i+1} \otimes \dots \otimes v_p.\]
Thus, if for $i \leq j$ we denote the element of $\bbb$ that maps $e_j$
to $e_i$ and everything else to zero by $E_{ij}$, then if $j_1 \leq i_1$,
$j_2 \leq i_2$, \dots, $j_p \leq i_p$
\[
 e_{j_1} \otimes e_{j_2} \otimes \dots \otimes e_{j_p} = (E_{j_1i_1},0,\dots,0)(0,E_{j_2i_2},0,\dots,0)\cdots(0,\dots,0,E_{j_pi_p})(v).
\]
It follows that indeed, $E \subseteq W$.  Since $W$ is by assumption
spanned by $k + 1$ such eigenvectors $v_i$, and hence $W = \sum_i \la
B^p v_i \ra$, we find that
\[ W = \sum_{a=1}^{k+1} F_a \]
where each $F_a$ is a tensor product of some of the $E_j$.  Now for a
fixed $a$, let $F_a = E_{i_1} \otimes \dots \otimes E_{i_p}$. Then most
$i_k$ must be equal to $1$: indeed, the dimension of $F_a$ is $(i_1+1)
(i_2+2) \cdots (i_p+1) \leq k + 1$, hence at most $\lfloor \log_2 (k+1)
\rfloor$ of the $i_j$ can be $1$ or larger. We conclude that there are
at most $p_1=(k+1)\lfloor \log_2(k+1) \rfloor$ indices where at least
one $F_a$ has a factor of dimension at least $2$. Since $p>p_1$, there
is at least one index $i$ where all $F_a$ have a one-dimensional factor
$E_0$. Thus, in the flattening corresponding to $\{ i\}$, $[p] - \{i\}$,
$W$ corresponds to a space of the form
\[ W = E_0 \otimes H \]
where $\dim H = \dim W$. Now any linear function on $V$ that is not zero
on $E_0$ yields a contraction of $W$ of dimension equal to $\dim H =
\dim W = k + 1$ -- a contradiction.
\end{proof}

\begin{proof}[Proof of Lemma~\ref{lem:ContractionPP_0}]
Let $p_0 = 2p_1 = 2(k+1)\lfloor \log_2(k+1) \rfloor$.  We will proceed by
induction on $p > p_0$. So let $v \in V^{\otimes p}$ be an element such
that for all subsets $I \subset [p]$ with $p-p_0$ elements, the image
of $v$ under any contraction $V^{\otimes p} \to V^{\otimes [p]-I} \cong
V^{\otimes p_0}$ is an element of $\Yp{p_0}$.  Note that if $p > p_0 +
1$, then for the image $v'$ of $v$ under any contraction $V^{\otimes p}
\to V^{\otimes [p]-\{i\}}$, any pure contraction $V^{\otimes [p]-\{i\}}
\to V^{\otimes [p]-\{i\}-I}$ takes $v'$ to an element of $\Yp{p_0}$,
where now $I$ ranges over all subsets with $p - 1 - p_0$ elements. As
a consequence, if $p > p_0 + 1$, then we may assume by induction that $v'$ is
an element of $\Yp{p-1}$. Thus, we need to show the following: Given $v
\in V^{\otimes p}$ such that all pure contractions of $v$ in $V^{\otimes
[p]-\{i\}}$ are elements of $\Yp{p-1}$, then $v \in \Yp{p}$, provided
$p > p_0$.

In order to see this, let $[p] = I \cup J$ be any partition and consider
the corresponding flattening
\[
 \flat \colon V^{\otimes p} \to V^{\otimes I} \otimes V^{\otimes J}.
\]
After exchanging the roles of $I$ and $J$ if necessary, we may assume
that $\vert J \vert  > p_1$. The statement that all $(k+1) \times
(k+1)$-subdeterminants on $\flat v$ are zero is equivalent to the
statement that same as saying that $\flat v$ has rank at most $k$ when
regarded as a linear map from $(V^{\otimes I})^*$ to $V^{\otimes J}$,
or, in other words, that the image $W \subseteq V^{\otimes J}$ of this
map has dimension at most $k$.

Since $\vert J \vert > p_1$ we may apply
Lemma~\ref{lem:SubspaceContraction} to $W$. Indeed, for each $j \in J$,
all contractions in factor $j$ of $V^{\otimes J}$ map $W$ to subspaces of
$V^{\otimes J-\{j\}}$ of dimension at most $k$. This follows form the
fact that this subspace is equal to the image $W'$ of the map $(V^{\otimes
I})^* \to V^{\otimes J-j}$ obtained by first applying $\flat v$ and then
contracting $V^{\otimes J} \to V^{\otimes J-i}$. This on the other hand,
is nothing but the map $\flat' (v')$ where $v'$ is the image of $v$
under the same contraction but applied to $V^{\otimes p}$, and $\flat'$
is the flattening of $[p]-\{j\}$ along $I,J-j$.  Since $v'$ gives rise
to a map of rank at most $k$ by assumption, $\dim W' \leq k$ as claimed.
Now this holds for all contractions and all factors and we may conclude
that, indeed, $\dim W \leq k$, and $\flat v$ has rank at most $k$.
\end{proof}

\section{Equivariantly Noetherian rings and spaces} \label{sec:EqNoeth}

We briefly recall the notions of equivariantly Noetherian rings and
topological spaces, and proceed to prove the main result of this section,
namely, that $\Yinf$ is $\Ginf$-Noetherian (Theorem~\ref{thm:YEqNoether}).

If a monoid $\Pi$ acts by means of endomorphisms on a commutative
ring $R$ (with $1$), then we call $R$ \emph{equivariantly Noetherian}, or
$\Pi$-Noetherian, if every chain $I_1 \subseteq I_2 \subseteq \ldots$
of $\Pi$-stable ideals stabilises. This is equivalent to the statement
that every $\Pi$-stable ideal in $R$ is generated by finitely many
$\Pi$-orbits. Similarly, if $\Pi$ acts on a topological space $X$ by
means of continuous maps $X \to X$, then we call $X$ equivariantly
Noetherian, or $\Pi$-Noetherian, if every chain $X_1 \supseteq X_2
\supseteq \ldots$ of $\Pi$-stable closed subsets stabilises. If $R$ is
an $K$-algebra then we can endow the set $X$ of $K$-valued points of $R$,
i.e., $K$-algebra homomorphisms $R \to K$ (sending $1$ to $1$), with the
Zariski topology. An endomorphism $\Psi: R \to R$ gives a continuous map
$\psi:X \to X$ by pull-back, and if $R$ has a left $\Pi$-action making
it equivariantly Noetherian, then this induces a right $\Pi$-action on
$X$ making $X$ equivariantly Noetherian. This means, more concretely,
that any $\Pi$-stable closed subset of $X$ is defined by the vanishing of
finitely many $\Pi$-orbits of elements of $R$. If $\Pi$ happens to be a
group, then we can make the right action into a left action by taking
inverses. Here are some further easy lemmas; for their proofs we refer
to \cite{Draisma08b}.

\begin{lm}\label{lm:Closed}
If $X$ is a $\Pi$-Noetherian topological space, then any 
$\Pi$-stable subset of $X$ is $\Pi$-Noetherian with respect to the
induced topology.
\end{lm}

We will only need this statement for closed subsets, for which the proof
is contained in \cite{Draisma08b}. For completeness, we prove the general
case here.

\begin{proof}
Let $Y$ be any $\Pi$-stable subset of $X$, and let $Y_1 \supseteq
Y_2 \supseteq \ldots$ be a sequence of $\Pi$-stable closed subsets of
$Y$. Define $X_i$ as the closure of $Y_i$ in $X$, so that $X_i \cap Y=Y_i$
by properties of the induced topology. A straightforward verification
shows that each $X_i$ is $\Pi$-stable, hence as $X$ is
$\Pi$-Noetherian we find that the sequence $X_1 \supseteq
X_2 \supseteq \ldots$ stabilises. But then the
sequence $Y_1=X_1 \cap Y \supseteq Y_2=X_2 \cap Y \supseteq
\ldots$ stabilises, as well.
\end{proof}

\begin{lm}\label{lm:Union}
If $X$ and $Y$ are $\Pi$-Noetherian topological spaces, then the disjoint
union $X \cup Y$ is also $\Pi$-Noetherian with respect to the disjoint
union topology and the natural action of $\Pi$.
\end{lm}

\begin{lm}\label{lm:Image}
If $X$ is a $\Pi$-Noetherian topological space, $Y$ is a topological
space with $\Pi$-action (by means of continuous maps), and $\psi:X \to Y$
is a $\Pi$-equivariant continuous map, then $\im \psi$ is $\Pi$-Noetherian
with respect to the topology induced from $Y$.
\end{lm}

\begin{lm} \label{lm:Quot}
If $\Pi$ is a group and $\Pi' \subseteq \Pi$ a subgroup acting from the
left on a topological space $X'$, and if $X'$ is a $\Pi'$-Noetherian,
then the orbit space $X:=(\Pi \times X')/\Pi'$ is a left-$\Pi$-Noetherian
topological space.
\end{lm}

In this lemma, $\Pi \times X'$ carries the direct-product topology
of the discrete group $\Pi$ and the topological space $X'$, the right
action of $\Pi'$ on it is by $(\pi,x')\sigma=(\pi \sigma, \sigma^{-1}
x)$, and the topology on the quotient is the coarsest topology
that makes the projection continuous. The left action of $\Pi$ on the
quotient comes from left-action of $\Pi$ on itself.  As a consequence,
closed $\Pi$-stable sets in $X$ are in one-to-one correspondence with
closed $\Pi'$-stable sets in $X'$, whence the lemma. Next we recall a
fundamental example of an equivariantly Noetherian ring, which will be
crucial in what follows.

\begin{thm}[\cite{Cohen67,Hillar09}] \label{thm:RlN}
For any Noetherian ring $Q$ and any $l \in \NN$, the ring $Q[x_{ij}
\mid i=1,\ldots,l;\ j=1,2,3,4,\ldots]$ is equivariantly Noetherian with
respect to the action of $\Inc(\NN)$ by $\pi x_{ij}=x_{i\pi(j)}.$
\end{thm}

Main Theorems I, II, and III will be derived from the following theorem,
whose proof needs the rest of this section.

\begin{thm} \label{thm:YEqNoether}
For every natural number $k$ the variety $\Yinf$ is a $\Ginf$-Noetherian
topological space.
\end{thm}

We will proceed by induction on $k$. For $k=0$ the variety
$\Yinf$ consists of a single point, the zero tensor, and the
theorem trivially holds.  Now assume that the theorem holds for
$k-1$. By Proposition~\ref{prop:YFinite} there exist $k$-tuples
$\bw_1,\ldots,\bw_N,\bw'_1,\ldots,\bw'_N$ of words with finite support,
such that each word in $\bw_a$ and each word in $\bw'_a$ have disjoint
supports and such that $\Yinf[k-1]$ is the common zero set of the
polynomials in $\bigcup_{a=1}^N \Ginf \det(x[\bw_a;\bw'_a])$. For each
$a=1,\ldots,N$ let $Z_a$ denote the open subset of $\Yinf$ where not
all elements of $\Ginf \det(x[\bw_a;\bw'_a])$ vanish; hence we have
\[ \Yinf = \Yinf[k-1] \cup Z_1 \cup \ldots \cup Z_N. \]
We will show that each $Z_a,\ a=1,\ldots,N$ is a $\Ginf$-Noetherian
topological space, with the topology induced from the Zariski topology on
$\Ainf$. Together with the induction hypothesis and Lemmas~\ref{lm:Union}
and \ref{lm:Image}, this then proves that $\Yinf$ is $\Ginf$-Noetherian,
as claimed.

To prove that $Z:=Z_a$ is $\Ginf$-Noetherian, consider
$\bw:=\bw_a=(w_1,\ldots,w_k)$ and $\bw':=\bw'_a=(w'_1,\ldots,w'_k)$. Let
$p$ be the maximum of the union of all supports of the $w_i$ and
the $w'_j$, and let $Z'$ denote the open subset of $\Yinf$ where
$\det(x[\bw;\bw'])$ is non-zero. This subset is stable under the group
$\Sinf'$ of all permutations in $\Sinf$ that restrict to the identity
on $[p]$.

\begin{lm} \label{lm:Zprime}
The open subset $Z' \subseteq \Yinf$ is an $\Sinf'$-Noetherian topological
space.
\end{lm}

\begin{proof}
We will prove that it is $\Inc(\NN)'$-Noetherian, where $\Inc(\NN)'$
is the set of all increasing maps $\pi \in \Inc(\NN)$ that restrict
to the identity on $[p]$; consult Section~\ref{sec:Infinite} for the
action of $\Inc(\NN)$. Since the $\Inc(\NN)'$-orbit of an equation is
contained in the corresponding $\Sinf'$-orbit, this implies that $Z'$
is $\Sinf'$-Noetherian.

We start with the polynomial ring $R$ in the variables $x_w$, where
$w$ runs over all infinite words over $\{0,1,\ldots,n-1\}$ with the
property that the support of $w$ contains at most $1$ position that
is larger than $p$. Among these variables there are $n^p$ for which
the support is contained in $[p]$, and the remaining variables are
labelled by $[n]^p \times \{1,\ldots,n-1\} \times (\NN-[p])$. On these
variables acts $\Inc(\NN)'$, fixing the first $n^p$ variables and acting
only on the last (position) index of the last set of variables. By
Theorem~\ref{thm:RlN} with $Q$ the ring in the first $n^p$ variables
and $l=n^p \times (n-1)$, the ring $R$ is $\Inc(\NN)'$-Noetherian.
Let $S=R[\det(x[\bw;\bw'])^{-1}$ be the localisation of $R$ at the
determinant $\det x[\bw,\bw']$; again, $S$ is $\Inc(\NN)'$-Noetherian. We
will construct an $\Inc(\NN)'$-equivariant map $\psi$ from the set of
$K$-valued points of $S$ to $\Ainf$ whose image contains $Z'$.  We do
this, dually, by means of an $\Inc(\NN)'$-equivariant homomorphism $\Psi$
from $\Tinf$ to $S$.

To define $\Psi$ recursively, we first fix a partition $I,J$ of $[p]$ such
that each $w_i$ has support contained in $I$ and each $w'_j$ has support
contained in $J$. Now if $x_w \in \Tinf$ is one of the variables in $R$,
then we set $\Psi x_w:=x_w$. Suppose that we have already defined $\Psi$
on variables $x_w$ such that $\supp(w)-[p]$ has cardinality at most $b$,
and let $w$ be a word for which $\supp(w)-[p]$ has cardinality $b+1$. Let
$q$ be the maximum of the support of $w$, and write $w=w_{k+1}+w'_{k+1}$,
where the support of $w'_{k+1}$ is contained in $J \cup \{q\}$ and the
support of $w_{k+1}$ is contained in $I \cup \{p+1,\ldots,q-1\}$. Consider
the determinant of the matrix
\[ x[w_1,\ldots,w_k,w_{k+1};w'_1,\ldots,w'_{k+1}]. \]
This determinant equals
\[ \det(x[w_1,\ldots,w_k;w'_1,\ldots,w'_k]) \cdot x_w - P, \]
where $P \in \Tinf$ is a polynomial in variables that are of the form
$x_{w_i+w'_j}$ with $i,j \leq k+1$ but not both equal to $k+1$. All of
these $w_i+w'_j$ have supports containing at most $b$ elements outside
the interval $[p]$, so $\Psi(P)$ has already been defined. Then we set
\[ \Psi x_w := \det(x[\bw,\bw'])^{-1} \psi(P). \]
The map $\Psi$ is $\Inc(\NN)'$-equivariant by construction.

The set $Z' \subseteq \Yinf$ is contained in the image of the map
$\psi$. Indeed, this follows directly from the fact that the determinant
of the matrix
\[ x[w_1,\ldots,w_k,w_{k+1};w'_1,\ldots,w'_{k+1}] \]
vanishes on $Z'$ while $\det(x[\bw,\bw'])$ does not. More precisely,
$Z'$ equals the intersection of $\Yinf$ with $\im \psi$, and hence by
Lemmas~\ref{lm:Image} and~\ref{lm:Closed} it is $\Inc(\NN)'$-Noetherian.
We already pointed out that this implies that $Z'$ is $\Sinf'$-Noetherian.
\end{proof}

\begin{re}
While it is possible to describe the map $\Psi$ in the proof of the theorem also by its action on points, it is rather involved and apparently yields no additional insight. A point worth mentioning, though, is the the following fact: if, for $i > p$, we denote the contraction $V^{\otimes [p]\cup\{i\}} \to V^{\otimes p}$ (along $x_0$ in the last factor) by $\partial_i$,  the ring $R$ mentioned in the proof is the coordinate ring of the affine scheme whose $K$-points are the (infinite) sequences
$(v_1,v_2,\dots)$ such that $v_i \in V^{\otimes [p]\cup\{i\}}$ and $\partial_i(v_i) = \partial_j(v_j)$ for all $i,j$, and $S$ is the coordinate ring of the open subset of sequences where a certain $k \times k$-subdeterminant is nonzero.
So $\Psi$ may be thought of as a ``glueing together'' of such a sequence to form an element in $\Ainf$.
\end{re}

\begin{ex}
The following shows that the map $\Psi$ in the proof above is not
$\Sinf'$-equivariant, which justifies the detour via $\Inc(\NN)'$. Take
$\bw,\bw'$ as in Example~\ref{ex:Det}, and take $w=0012$. The matrix
in the proof above equals
\[
\begin{bmatrix}
x_{11} & x_{12} & x_{1002}\\
x_{21} & x_{22} & x_{2002}\\
x_{011} & x_{021} & x_{0012}
\end{bmatrix}\]
and gives rise to
\[ \Psi x_{0012}=(x_{11}x_{22}-x_{12}x_{21})^{-1}
	\cdot (x_{2002}(x_{11}x_{021}-x_{12}x_{011})-
		x_{1002}(x_{21}x_{021}-x_{22}x_{011})). \]
On the other hand, take $w=0021$, obtained from $w$ by permuting
positions $3$ and $4$. The matrix in the proof above equals
\[
\begin{bmatrix}
x_{11} & x_{12} & x_{1001}\\
x_{21} & x_{22} & x_{2001}\\
x_{012} & x_{022} & x_{0021}
\end{bmatrix}
\]
and gives rise to
\[ \Psi x_{0021}=(x_{11}x_{22}-x_{12}x_{21})^{-1}
	\cdot (x_{2001}(x_{11}x_{022}-x_{12}x_{012})-
		x_{100}(x_{21}x_{022}-x_{22}x_{012})), \]
which is {\em not} obtained from the expression above by permuting
positions $3$ and $4$.
\end{ex}

Now that $Z'$ is $\Sinf'$-Noetherian, Lemma~\ref{lm:Quot} implies that
the $\Ginf$-space $(\Ginf \times Z')/\Sinf$ is $\Ginf$-Noetherian. The
map from this space to $\Ainf$ sending $(g,z')$ to $gz'$ is
$\Ginf$-equivariant and continuous, and its image is the open set $Z
\subseteq \Yinf$. Lemma~\ref{lm:Image} now implies that $Z$ is
$\Sinf$-Noetherian. We conclude that, in addition to the closed subset
$\Yinf[k-1] \subseteq \Yinf$, also the open subsets $Z_1,\ldots,Z_N$
are $\Sinf$-Noetherian. As mentioned before, this implies that
$\Yinf=\Yinf[k-1] \cup Z_1 \cup \ldots \cup Z_N$ is
$\Sinf$-Noetherian, as claimed in Theorem~\ref{thm:YEqNoether}.

\section{Proofs of main theorems} \label{sec:Proofs}

We are now in a position to prove our Main Theorems I, II, and III. We
start with the latter.

\begin{proof}[Proof of Main Theorem III.]
As $\Xinf$ is a closed $\Ginf$-stable subset of $\Yinf$, and as $\Yinf$
is a $\Ginf$-Noetherian topological space (Theorem~\ref{thm:YEqNoether}),
$\Xinf$ is cut out from $\Yinf$ by finitely many $\Ginf$-orbits of
equations. Moreover, $\Yinf$ itself is cut out from $\Ainf$ by finitely
many $\Ginf$-orbits of Equations (Proposition~\ref{prop:YFinite}),
and hence the same is true for $\Xinf$.
\end{proof}

In the proofs of Main Theorems I and II we will use inclusion maps
\[ \tau_p: V^{\otimes p} \to V^{\otimes p+1},\ \omega \mapsto \omega
\otimes e_0, \]
where $e_0$ is an element of $V$ such that $x_0(e_0)=1$.  This map sends
$\Xp{p}$ into $\Xp{p+1}$ and satisfies $\pi_p \circ
\tau_p=\id_{V^{\otimes p}}$.

\begin{proof}[Proof of Main Theorem I.]
By Lemma~\ref{lm:AllSame} it suffices to prove that for fixed $k \in \NN$
there exists a $d \in \NN$ such that for all $p \in \NN$ the variety
$\Xp{p}$ of tensors of border rank at most $k$ is defined in $V^{\otimes
p}$ by polynomials of degree at most $d$. By Main Theorem III there
exists a $d$ such that $\Xinf$ is defined in $\Ainf$ by polynomials
of degree at most $d$; we prove that the same $d$ suffices in Main
Theorem I. Indeed, suppose that all polynomials of degree at most $d$
in the ideal of $\Xp{p}$ vanish on a tensor $\omega \in V^{\otimes
p}$. Let $\omega_\infty$ be the element of $\Ainf$ obtained from
$\omega$ by successively applying $\tau_p,\tau_{p+1},\ldots$. More
precisely, for any coordinate $x_w \in T_q,\ q \geq p$ we have
$x_w(\omega_\infty)=x_w(\tau_{q-1}\tau_{q-2}\cdots\tau_p \omega)$, and
this determines $\omega_\infty$. We claim that $\omega_\infty$ lies in
$\Xinf$. Indeed, otherwise some $T_q$ contains a polynomial $f$ of degree
at most $d$ that vanishes on $\Xp{q}$ but not on $\omega_\infty$. Now
$q$ cannot be smaller than $p$, because then $f$ vanishes on $\Xp{p}$
but not on $\omega$. But if $q \geq p$, then $f \circ \tau_{q-1} \circ
\cdots \circ \tau_p$ is a polynomial in $T_p$ of degree at most $d$
that vanishes on $\Xp{p}$ but not on $\omega$. This contradicts the
assumption on $\omega$.
\end{proof}

The proof of Main Theorem II is slightly more involved.

\begin{proof}[Proof of Main Theorem II]
By Lemma~\ref{lm:AllSame} it suffices to show that for fixed $k \in \NN$
there exists a $p_0$ such that a tensor in $V^{\otimes p},\ p \geq p_0$
is of border rank at most $k$ as soon as all its contractions along pure
$(p-p_0)$-tensors to $V^{\otimes p_0}$ have border rank at most $k$.
By Main Theorem III there exists a $p_0$ such that the $\Ginf$-orbits
of the equations of $\Xp{p_0}$ define $\Xinf$. We claim that this value
of $p_0$ suffices for Main Theorem II, as well. Indeed, suppose that
$\omega \in V^{\otimes p}$ has the property that all its contractions
along pure tensors to $V^{\otimes p_0}$ lie in $\Xp{p_0}$, and construct
$\omega_\infty \in \Ainf$ as in the proof of Main Theorem I. We claim
that $\omega_\infty$ lies in $\Xinf$. For this we have to show that for
each $f$ in the ideal of $\Xp{p_0}$ and each $g \in \Ginf$ the polynomial
$gf$ vanishes on $\omega_\infty$. Let $q \in \NN$ be such that $g \in
G_q$. By construction, $f \in T_{p_0}$ is identified with the function
in $T_q$ obtained by precomposing $f$ with the contraction
$V^{\otimes q} \to V^{\otimes p_0}$ along the pure tensor $x_0^{\otimes
q-p_0}$ on the last $q-p_0$ factors. Hence $gf$ is the
same as contraction $V^{\otimes q} \to V^{\otimes p_0}$ along {\em some}
pure tensor (in some of the factors), followed by $g'f$ for some $g'
\in G_{p_0}$. Evaluating $gf$ at the tensor $\omega_\infty$ is the same
as evaluating it at
\[ \omega \otimes (e_0)^{\otimes q-p}, \]
and boils down to contracting some, say $l$, of the factors $e_0$ with
a pure tensor in $(V^*)^{\otimes l}$, and $q-p_0-l$ of the remaining
factors $V$ with a pure tensor $\mu$ in $(V^*)^{\otimes q-p_0-l}$, and
evaluating $g'f$ at the result. But this is the same thing as contracting
$\omega$ with $\mu$ in $(V^*)^{\otimes q-p_0-l}$ to obtain a $\omega'
\in V^{\otimes p-q+p_0+l}$ and evaluating $g'f$ at $\omega' \otimes
e_0^{q-p-l}$. Now by assumption $\omega'$ lies in $\Xp{p-q+p_0+l}$
(since $p-q+p_0+l \leq p_0$), and hence $\omega' \otimes e_0^{q-p-l}$
lies in $\Xp{p_0}$. This proves that $g'f$ vanishes on it, so that $gf$
vanishes on $\omega_\infty$, as claimed. Hence $\omega_\infty$ lies in
$\Xinf$. But the projection $\Ainf \to V^{\otimes p}$ sends
$\omega_\infty$ to $\omega$ and $\Xinf$ to $\Xp{p}$. Hence $\omega
\in \Xp{p}$, as required.
\end{proof}

\section{All equations?} \label{sec:Ideal}

In contrast with Main Theorem III, we do not believe that the full ideal
of $\Xinf$ is generated by finitely many $\Ginf$-orbits; and in fact this
fails already if $k = 1$ where $\Xinf[1] = \Yinf[1]$.  Since it takes
some space, the proof of this negative result is deferred to the appendix.
To obtain finiteness results for the full ideal of $\Xinf$ we propose to
use further non-invertible symmetries, provided by the following monoid.

\begin{de}
The {\em substitution monoid} $\Subs(\NN)$ consists of infinite sequences
$\sigma=(\sigma(1),\sigma(2),\ldots)$ of finite, pairwise disjoint subsets
of $\NN$. The multiplication in this monoid is given by
\[ (\sigma \pi)(p)=\bigcup_{q \in \pi_p}\sigma(q). \]
\end{de} 

A straightforward computation shows that $\Subs(\NN)$ is indeed an
associative monoid with unit element $1=(\{1\},\{2\},\{3\},\ldots)$.
It has a natural left action on infinite words over
$\{0,1,\ldots,n-1\}$ with finite support: $\sigma w:=w'$ where $w'(p)=0$
if $p$ does not lie in any of the $\sigma(q)$, and $w'(p)=w(q)$ if $p$
lies in $\sigma(q)$. We let $\Subs(\NN)$
act on $\Tinf$ by algebra homomorphisms determined by $\sigma x_w :=
x_{\sigma w}$.  This gives rise to a {\em right} action of $\Subs(\NN)$
on $\Ainf$.  The following lemma explains our interest in this monoid
of non-invertible symmetries.

\begin{lem}
 The ideal of $\Xinf$ is stable under the $\Subs(\NN)$-action.
\end{lem}

\begin{proof}
First, for pure tensors, a straightforward calculation shows that for
$\bw=(w_1,w_2)$ and $\bw'=(w_1',w_2')$ pairs of words as before, we have
$\sigma \det x[\bw,\bw']=\det x[\sigma \bw,\sigma \bw']$ where $\sigma$
acts component-wise on $\bw$ and $\bw'$. Hence $\Subs(\NN)$ stabilises
$\Xinf[1]$. Since $\Subs(\NN)$ acts linearly on $\Ainf$, this means
that $\Subs(\NN)$ also stabilises the set of sums of $k$ elements of
$\Xinf[1]$. But then it also stabilises the Zariski-closure $\Xinf$
of this set.
\end{proof}

We have the following conjecture. 

\begin{conj} \label{conj:Subs}
For every fixed $k$, the ideal of $\Xinf$ is generated by finitely many
$\Subs(\NN)$-orbits of equations.
\end{conj}

This conjecture certainly holds for $k=1$ and $k=2$: then the ideal
is generated by subdeterminants of flattenings, which by the following
lemma form a single $\Subs(\NN)$-orbit.

\begin{lem}\label{lem:YINF_SUBS}
For fixed $k$, the $k \times k$-sub-determinants $\det x[\bw,\bw']$
in $\Tinf$ form a single $\Subs(\NN)$-orbit.
\end{lem}

\begin{proof}
Let us write the elements of $\bw = (w_1,w_2,\dots,w_k)$ and $\bw' =
(w_1',w_2',\dots,w_k')$ in a $2k \times \NN$-table $T_{\bw,\bw'}$ with
the entry at $(i,j)$ equal to $w_i(j)$ if $i \leq k$ and equal to $w'_{i-k}(j)$
if $k < i \leq 2k$. Note that we require that each $w_i$ has support
disjoint from that of each $w_j'$, so that each column either starts
with $k$ zeroes or ends with $k$ zeroes. Fix $\bu=(u_1,\ldots,u_k)$
and $\bu'=(u'_1,\ldots,u'_k)$ such that in the corresponding table
$T_{\bu,\bu'}$ every non-zero column starting or ending with $k$ zeroes
occurs exactly once (hence the number of non-zero columns is
$2(n^k-1)$). For general $\bw,\bw'$ as above, define $\sigma
\in \Subs(\NN)$ as follows: $\sigma(j):=\emptyset$ if the $j$-th column
of $T_{\bu,\bu'}$ is zero, and
\[ \sigma(j):=\{l \in \NN \mid \text{ the $l$-th column of $T_{\bw,\bw'}$
equals the $j$-th column of $T_{\bu,\bu'}$}\} \]
otherwise. By construction we find $\sigma u_i=w_i$ and $\sigma
u_i'=w_i'$ for all $i$, and therefore $\sigma \det x[\bu,\bu']=\det
x[\bw,\bw']$. 
\end{proof}

We point out two interesting submonoids of $\Subs(\NN)$. First,
the submonoid of $\Subs(\NN)$ consisting of $\sigma$ for which each
$\sigma(i)$ is non-empty. This submonoid acts by means of injective
endomorphisms on $\Tinf$, and satisfies a variant of the previous lemma
with ``a single orbit'' replaced by ``finitely many orbits''. The same
is true for an even smaller submonoid, denoted $\Subs_{<}(\NN)$ where we
require $\sigma(i) \neq \emptyset$ for all $i$ {\em and} $\max \sigma(1)
< \max \sigma(2) < \ldots$---again at the cost of increasing the number
of orbits.  Using this submonoid we prove the following finiteness result.

\begin{prop}
For every natural number $d$, the polynomials in $\Tinf$ of
degree at most $d$ form a Noetherian $\Subs(\NN)$-module, that is,
every $\Subs(\NN)$-submodule is generated by a finite number of
$\Subs(\NN)$-orbits.
\end{prop}

We prove the stronger statement, where $\Subs(\NN)$ is replaced by the
smaller monoid $\Subs_<(\NN)$. 

\begin{proof}
For any natural number $m$, define a partial order on infinite words over $\{0,1,\ldots,m\}$
of finite support by $w \leq w'$ if and only if there exists a
$\sigma \in \Subs_<(\NN)$ such that $\sigma w=w'$. By induction
on $m$, we prove that this is a well-quasi-order in the sense of
\cite{Kruskal72}. For $m=0$ there is only one word, and $\leq$ is
trivially a well-quasi-order. Suppose next that it is a well-quasi-order
for $m-1$, and consider a sequence $w_1,w_2,\ldots$ of words over
$\{0,1,\ldots,m\}$ of finite support. After passing to a subsequence we
may assume that all $w_a$ use exactly the same alphabet. If this alphabet
has less than $m+1$ elements, then by the induction hypothesis there exist
$a<b$ such that $w_a \leq w_b$, and we are done. Hence we may assume that
the common alphabet of all words is $\{0,1,\ldots,m\}$. For each $w_a$,
let $s_a$ be the symbol in $w_a$ whose {\em last} occurrence precedes
the last occurrences of all other symbols, that is, the tail of $w_a$
after the last occurrence of $s_a$ uses an alphabet of cardinality $m-1$.
By passing to a subsequence we may assume that all $s_a$ are equal,
and without loss of generality that they are all equal to $m$. 

Now there is a weaker partial order on words of finite support over the
alphabet $\{0,\ldots,m\}$, defined by $w \leq' w'$ if and only if there
exists a $\pi \in \Inc(\NN)$ such that the symbol at position $j$ of $w$
equals the symbol at position $\pi(j)$ of $w'$.  This partial order {\em
is} a well-quasi-order by Higman's lemma \cite{Higman52}. Hence after
passing to a subsequence we may assume that $w_1 \leq' w_2 \leq' \ldots$.

Returning to our main argument, for each $w_a$ let $w_a'$ denote the
word obtained from it by deleting the head up to and including the
last occurrence of $s_a=m$, so $w_a'$ is a word that uses exactly the
alphabet $\{0,\ldots,m-1\}$ ($w_a'$ is indexed by $1,2,3,\ldots$).
By the induction hypothesis, there exist $a<b$ such that $w_a' \leq
w_b'$. We claim that then also $w_a \leq w_b$. We construct $\sigma \in
\Subs_<(\NN)$ witnessing this fact as follows. Let $\pi \in \Inc(\NN)$
be a witness of $w_a \leq' w_b$, i.e., for all $j \in \NN$ we have
$(w_a)(j)=(w_b)(\pi(j))$; let $j_a,j_b$ be the positions of the last
occurrences of $m$ in $w_a,w_b$, respectively; and let $\gamma \in
\Subs_<(\NN)$ be a witness of $w_a' \leq w_b'$.  Then define $\sigma$ by
\[ 
\sigma(j)=\begin{cases}
	\{\pi(j)\} \text{ if } j<j_a,\\

	\{j_b\} \cup \{i \in [j_b-1] \mid w_b(i)=m 
	\text{ and } i \not \in \pi[j_a-1]\} \text{ if }
	j=j_a, \\

	\{j_b+i \mid i \in \gamma(j-j_a)\}
	\cup 
	\{i \in [j_b-1] \mid w_b(i)=w_a(j)
	\text{ and } i \not \in \pi[j_a-1] \} \\
	\quad \text{if $j>j_a$ and $j-j_a$ is the first position of
	$w_a(j)$ in $w'_a$, and}\\

	\{j_b+i \mid i \in \gamma(j-j_a)\}\\
	\quad \text{if $j>j_a$ and $j-j_a$ is not the first position
	of $w_a(j)$ in $w'_a$.}
\end{cases}
\]
A straightforward computation---see below for an example---shows that
$\sigma w_a = w_b$, so that $w_a \leq w_b$, as required. This shows that
$\leq$ is indeed a well-quasi-order.

We use this to prove that for every infinite sequence $u_1,u_2,\ldots$
of monomials in the $x_w$ of degree at most $d$ there exists $a<b$ and
$\sigma \in \Subs_{<}(\NN)$ such that $\sigma u_a=u_b$. Since there exists
a degree $e \leq d$ such that infinitely many of the $u_a$ have degree
equal to $e$, we may assume that all $u_a$ have the same degree $e$. To
each $u_a$ we then associate an $e \times \NN$ table $T_a$, whose rows
record the words $w$ corresponding to variables $x_w$ occurring in $u_a$,
including their multiplicities (in any order). The columns of $T_a$ form
a word of finite support over the alphabet $\{0,1,\ldots,n-1\}^e$. Hence
by the above, there exist $a<b$ and $\pi \in \Subs_<(\NN)$ such that
$\pi T_a=T_b$. But then also $\pi u_a=u_b$, as claimed. 

Now finally let $W$ be any $\Subs_<(\NN)$-submodule of the space of
polynomials in the $x_w$ of degree at most $d$. Choose any well-order
on monomials in the $x_w$ that is preserved by the $\Subs_<(\NN)$---for
instance, lexicographic with respect to the order on the variables defined
by $x_w \geq x_{w'}$ if $\sum_{j=1}^\infty w(j)n^j \geq \sum_{j=1}^\infty
w'(j) n^j$ (an inequality between the numbers represented by
$w$ and $w'$ in the $n$-ary number system). By the above, there exist $w_1,\ldots,w_l \in W$ such that
for all $w \in W \setminus \{0\}$ there exist $i\in [l]$ and $\pi \in
\Subs_<(\NN)$ for which the leading monomial of $\pi w_i$ (with respect
to the well-order) equals the leading monomial of $w$. This implies that
the $\Subs_<(\NN)$-orbits of $w_1,\ldots,w_l$ span $W$.
\end{proof}

\begin{ex}
Here is an example of the construction of $\sigma$ in the proof above.
Let $m=2$ and 
\begin{align*}
w_a&=0010212011 & j_a&=7 & w_a'&=011\\
w_b&=0020102012001011 & j_b&=10 & w_b'&=001011\\
\pi&=(1,2,5,6,7,9,10,11,13,15,17,18,19,\ldots)\\
\gamma&=(\{1,2,4\},\{5\},\{3,6\},\{7\},\{8\},\ldots)
\end{align*}
so that $\gamma w_a'=w_b'$ and $w_a(i)=w_b(\pi(i))$ for all $i$ (of
course $\gamma$ and $\pi$ with these properties are in general not
unique). Then $\sigma$ above equals
\[ \sigma=(\{1\},\{2\},\{5\},\{6\},\{7\},\{9\},\{3,10\},
\{4,8,11,12,14\},\{15\},\{13,16\},\{17\},\{18\},\ldots). \]
Note how all the gaps in $\im \pi$ are filled by the second half
of $\sigma$.
\end{ex}

We now have the following immediate corollary, which makes our conjecture
more plausible.  In fact, Andrew Snowden pointed out to us that his
beautiful theory of $\Delta$-modules \cite{Snowden10} yields a similar
equivalence.

\begin{cor}
Conjecture~\ref{conj:Subs} is equivalent to the statement that the ideal
of $\Xinf$ is generated by polynomials of bounded degree. \qed
\end{cor} 

Finally, we point out that the action of $\Subs_<(\NN)$ on the coordinate
ring $\Tinf$ allows for {\em equivariant Gr\"obner Basis} methods as in
\cite{Brouwer09e}. We conjecture that for a fixed value of $k$ (and of
$n>k$) Conjecture can in principle be proved by a finite equivariant
Gr\"obner basis computation. It would be interesting to try and do so
for the GSS-conjecture, even if the computer proof would not be nearly
as elegant as Raicu's proof.

\appendix

\section{The ideal of $\Xinf[1]$ is not finitely generated}
As announced in Section~\ref{sec:Ideal}, we will now show that the
ideal of $\Xinf[1]$ is not $\Ginf$-finitely generated, at least in
characteristic zero.
We write $\Tinf^i$ for the space of elements of
degree $i$ in $\Tinf$, and throughout $K$ is algebraically closed of characteristic zero.
Let $H_p \subset G_p$ be the connected component of
$1$ (ie.\ $H_p = \GL(V)^p$), and let $\Hinf \subseteq \Ginf$ be their union.
Similarly $\Tinf$ denotes the union of all $D^p$ where $D$ is the group of
diagonal matrices in $H_p$, and finally, $\Binf$ is the union of all $B^p$ with
$B$, contrary to before, the group of \emph{lower} triangular matrices in $H$.
In order to use representation theory, we need to use lower triangular matrices
to make sure that the embedding $T_p \mapsto T_{p+1}$ (given by tensoring with
$x_0$) maps eigenvectors of $B_p$ to eigenvectors for $B_{p+1}$. (In fact the
union of upper triangular matrices in $\Hinf$ has no eigenvectors at all in
$\Tinf$.)  
Next, let $X = X(\Dinf)$ be the character group of $\Dinf$ which we may
identify with $X(D)^\NN$ (and recall that $X(D) = \ZZ^n$, generated by
the weights of the $e_i$, which we will denote by $\varepsilon_i$).
We call a $\Hinf$-representation $M$ a \emph{lowest weight module} of weight $\lambda \in
X$, if $M$ is generated by an eigenvector of $\Binf$ that has $\Dinf$-weight
$\lambda$.
For example, $\Tinf^1$ is itself a lowest weight module with lowest
weight $(\varepsilon_0,\varepsilon_0,\dots)$ and eigenvector $x_0$ for
$\Binf$.

If $M$ is any representation for $\Hinf$, we write $M[k]$ for the
representation obtained by shifting $k$-places. That is, $h =
(h_1,h_2,\dots)\in \Hinf$ acts on $m$ as $(h_{k+1},h_{k+2},\dots)m$.
Similarly, for a weight $\lambda \in X$, we define $\lambda[k]$ to be
the weight obtained from $X$ by inserting $k$ zeros in the beginning.
If $M$ is a lowest weight module with lowest weight $\lambda$, then so
is $M[k]$ and its weight is $\lambda[k]$.

For the purposes of this appendix, an \emph{ind-finite representation} $W$ of $\Hinf$
is a directed union $W= \bigcup_p W_p$, where $W_1 \subseteq W_2 \subseteq
\dots $ are nested finite-dimensional vector spaces and $W_p$ is a $H_p$-representation such that the inclusions $W_p \to W_{p+1}$ are $H_p$-equivariant, and moreover, 
any $B^p$-eigenvector is mapped to a $B^{p+1}$-eigenvector. 

\begin{lem}
Let  $W$ be an ind-finite $\Hinf$-representation. Then $W$ is semi-simple, i.e.\ generated by irreducible subrepresentations, and each simple submodule is a lowest weight module with a uniquely determined lowest weight.
\end{lem}
\begin{proof}
	This is an immediate consequence of the same fact that $W$ is
	ind-finite and the lemma is true for finite dimensional
	representations of $H_p$. We skip the details.
\end{proof}

Let $\lambda_0 = (2\varepsilon_0,2\varepsilon_0,\dots)\in X$. It is the
lowest weight of the span of squares of pure tensors in $\Tinf^2$ (in
the finite dimensional setting this would be called the
Cartan-component). We denote this sub-representation by $C$.

For any lowest  weight $\lambda$ we define its \emph{complexity} as the
number of positions where it differs from $\lambda_0$. For all lowest
weights in $\Tinf^2$, this is finite.
The group 
$\Sinf$ acts on $\Hinf$ by permuting factors and on $X$ by permuting
entries. Consequently, if $M
\subseteq \Tinf^2$ is a lowest-weight module of weight
$\lambda$, then
$\pi M$ is a lowest weight module of weight $\pi \lambda$; in
particular, this $\Sinf$-action preserves the complexity of these
weights.

Now let $x_vx_w$ be any monomial of degree $2$ and let  $M$ be the
$\Hinf$-module generated by $x_vx_w$.

\begin{lem}\label{lem:BOUNDED}The complexity of every lowest weight appearing in $M$ is at most
the number of positions where the words $v$ and $w$ differ. In
particular, it is bounded. 
\end{lem}
\begin{proof}
Any lowest weight vector in $M$ is in the image of a lowest weight
vector under the multiplication map $m\colon \Tinf^1 \otimes \Tinf^1 \to
\Tinf^2$ and in fact this vector may be assumed to be an element of the
module generated by $x_v \otimes x_w$. Next, the claim
is invariant under the $\Sinf$-action, so we may assume that $v,w$
differ precisely at the first $k$ positions and then coincide. 

The space 
$\Tinf^1$ is a direct limit of the tensor products $(V^*)^{\otimes p}$; 
after reordering we get that $\Tinf^1\otimes \Tinf^1$ is isomorphic (as an
$\Hinf$-representation) to \[(V^{*\otimes k}\otimes V^{*\otimes k}) \otimes
(\Tinf^1[k]\otimes \Tinf^1[k])\]
where $\Hinf$ acts by its first $k$ entries (diagonally) on $V^{*\otimes k}\otimes
V^{*\otimes k}$.
With this notation in place, $x_v \otimes x_w$ has the form $(u_1\otimes
u_2) \otimes (
r \otimes r)$ where $u_1,u_2 \in V^{*\otimes k}$ and $r$ is a pure tensor
in $\Tinf^1[k]$; hence this is also true for any $\Hinf$-translate of
$x_v \otimes x_w$. Here we call ``pure tensor'' an element that is the image
of a pure tensor in one of the $V^{*\otimes p}$.

Note that $r \otimes r$ in $\Tinf^1[k] \otimes \Tinf^1[k]$ generates an
irreducible $\Hinf$-representation of weight $\lambda_0[k]$, since $r\otimes r$ is in the
orbit of $x_0\otimes x_0$, and thus, a representation isomorphic to
$C[k]$ (under the map $m$). It follows that $M$ is isomorphic to a
sub-representation of $(V^{*\otimes k}\otimes V^{*\otimes k}) \otimes
C[k]$. Any lowest weight vector here is the tensor product of one in
$(V^{*\otimes k})\otimes V^{*\otimes k})$ and the one in $C[k]$. From
this it follows easily that any lowest weight in the module generated by
$x_v\otimes x_w$ (and hence $M$) has the form $\omega + \lambda_0[k]$
where $\omega$ is a lowest weight of
$V^{*\otimes k}\otimes V^{*\otimes k}$ (and hence corresponds to just $k$
weights of $D$ and all zeros afterwards).
It follows that the complexity is at most $k$ as claimed.
\end{proof}
The crucial observation is now the following lemma.
\begin{lem}
The weights of the lowest weight modules in  $\Tinf^2$ have 
arbitrarily high complexity.
In other words, for each integer $N$, there exits a lowest weight of complexity at least $N$.
\end{lem} 
\begin{proof} Note that $\Tinf^2$ is equivariantly isomorphic to the second
	symmetric power $S^2(\Tinf^1)$.  For any $k > 0$, the canonical map
	$S^2(V^{*\otimes k}) \otimes C[k] \to \Tinf^2$ is an equivariant
	embedding where $\Hinf$ acts on $S^2(V^{*\otimes k})$ by means of the
	projection $\Hinf\to H_{k}$.

Standard representation theory for $H_k$ shows that $S^2(V^{*\otimes k})$
contains an  $H_{k}$-sub-representation isomorphic to $(\bigwedge^2
V^*)^{\otimes k}$ whenever $k$ is even. To see this, we observe that on on $W =
V^{*\otimes k}\otimes V^{*\otimes k}$ the group $S_{2k}$ acts by permuting the
tensor factors.  This allows us to identify $S^2(V^{*\otimes k})$ with the
symmetric tensors in $W$ (those fixed by the product of $k$-transpositions
$(1\, k+1)(2\, k+2)\cdots (k-1\, 2k)$).  Then $(\bigwedge^2V^*)^{\otimes k}$ is
isomorphic to the subrepresentation given by those tensors $w$ in $W$ that
satisfy $(i \, k+i)w=-w$ for all $i$.  If $k$ is even, any such tensor is
symmetric.  The lowest weight for $B^k$ of this representation is
$(\varepsilon_0+\varepsilon_1,\dots,\varepsilon_0+\varepsilon_1)$ ($k$
entries). So $(\bigwedge^2 V^*)^{\otimes k} \otimes C[k]$ is an irreducible
representation of $\Tinf^2$ and its lowest weight is \[
(\underbrace{\varepsilon_0+\varepsilon_1,\varepsilon_0+\varepsilon_1,\dots,\varepsilon_0
+ \varepsilon_1}_k,2\varepsilon_0,2\varepsilon_0,\dots).  \] The complexity of
this weight is $k$. So for any $N$, there is a lowest weight of complexity $N$
or $N + 1$.  \end{proof} As a corollary we obtain the following result:
\begin{prop} $\Tinf^2$, as a $\Ginf$-module, is not finitely generated.
	Moreover, the ideal of $\Xinf[1] = \Yinf[1]$ is not $\Ginf$-finitely
	generated.  
\end{prop} 
\begin{proof} The $\Ginf$-module generated by finitely many elements of
	$\Tinf^2$ will be contained in the module generated by all
	monomials $x_vx_w$ appearing in any of the generators.  Since
	$\Sinf$ preserves the complexity of lowest weights, any
	appearing irreducible $\Hinf$-module will have a lowest weight
	of complexity at most equal to the highest number of places
	where the two words $v,w$ differ in any of the generating
	monomials by Lemma~\ref{lem:BOUNDED}. Since $\Tinf^2$ contains
	lowest weight module of arbitrary complexity, this proves the
	first claim. 
	
	As for the second, if the ideal $I$ of $\Xinf[1]$ were
	$\Ginf$-finitely generated it could  be generated  by finitely
	many  homogeneous
	elements of degree $2$. It therefore suffices to show that $I \cap
	\Tinf^2$
	contains lowest weight modules of arbitrarily high complexity.
	In fact, as in the finite dimensional case, every lowest weight
	module other than $C$ will be contained in the ideal: if $f$ is
	a lowest weight vector of weight $\lambda$, say. Then there is
	$p$ such that $f \in T_p$ and hence $f \in I$ if and only if $f$
	vanishes on $\Xp[1]{p}$. Set $v:= e_0^{\otimes p}$, the highest weight
	vector in $V^{\otimes p}$.
	Since $\lambda \neq 2\varepsilon_0$, $f \neq x_0^2$ and so $f(v)
	= 0$. But then also $f(bv) = \lambda(b)^{-1}f(v) = 0$ for all $b
	\in B^p$. Since $B^pv$
	is open and dense in $\Xp[1]{p}$, it follows that $f \in I$ and
	we are done. 

	In fact, this shows that as in the case of $p$ finite, the ideal spans a
	complement to $C$.
\end{proof}

%\bibliographystyle{plain}
%\bibliography{diffeq,draismajournal}

\end{document}